\documentclass[12pt]{amsart}
\usepackage{amsmath}
\usepackage{amsfonts, amssymb, euscript, mathrsfs}
\usepackage{amsthm, upref}
\usepackage{graphicx}
\usepackage[usenames, dvipsnames]{color}

\usepackage{tikz}
\usepackage{enumerate}
\usepackage[margin=0.88in]{geometry}
\usepackage{aliascnt}
\usepackage{hyperref}
\usepackage{verbatim}

\allowdisplaybreaks[4]

\newcommand{\vk}{\varkappa}

\newcommand{\BR}{\mathbb{R}}

\newcommand{\al}{\alpha}

\newcommand{\de}{\delta}

\newcommand{\La}{\Lambda}
\newcommand{\ME}{\mathbf E}

\newcommand{\CF}{\mathcal F}

\newcommand{\MP}{\mathbf P}

\newcommand{\CA}{\mathcal A}
\newcommand{\CL}{\mathcal L}

\newcommand{\CN}{\mathcal N}

\newcommand{\Oa}{\Omega}
\newcommand{\oa}{\omega}
\newcommand{\si}{\sigma}

\newcommand{\pa}{\partial}
\renewcommand{\phi}{\varphi}

\newcommand{\eps}{\varepsilon}
\newcommand{\la}{\lambda}
\newcommand{\Ra}{\Rightarrow}

\newcommand{\ol}{\overline}

\newcommand{\CM}{\mathcal M}

\newcommand{\norm}[1]{\lVert#1\rVert}
\renewcommand{\comment}[1]{}
\newcommand{\CD}{\mathcal D}

\newcommand{\md}{\mathrm{d}}

\newcommand{\ml}{\mathfrak l}

\DeclareMathOperator{\Exp}{Exp}

\DeclareMathOperator{\TV}{TV}

\begin{document}

%\large

\theoremstyle{plain}
\newtheorem{thm}{Theorem}[section]
\newtheorem*{thmnonumber}{Theorem}
\newtheorem{lemma}[thm]{Lemma}
\newtheorem{prop}[thm]{Proposition}
\newtheorem{cor}[thm]{Corollary}
\newtheorem{open}[thm]{Open Problem}

\theoremstyle{definition}
\newtheorem{defn}{Definition}
\newtheorem{asmp}{Assumption}
\newtheorem{notn}{Notation}
\newtheorem{prb}{Problem}

\theoremstyle{remark}
\newtheorem{rmk}{Remark}
\newtheorem{exm}{Example}
\newtheorem{clm}{Claim}

\author{Andrey Sarantsev}

\title[Reflected Jump-Diffusions]{Explicit Rates of Exponential Convergence\\ for Reflected Jump-Diffusions on the Half-Line}

\address{Department of Statistics and Applied Probability, University of California, Santa Barbara}

\email{sarantsev@pstat.ucsb.edu}

\date{November 14, 2016. Version 21}

\keywords{Lyapunov function, stochastically ordered process, stochastic domination, reflected diffusion, reflected jump-diffusion, uniform ergodicity, exponential rate of convergence, competing Levy particles, Levy process, gap process, jump measure, reflected Levy process}

\subjclass[2010]{Primary 60J60, secondary 60J65, 60H10, 60K35}

\begin{abstract} Consider a reflected jump-diffusion on the positive half-line. Assume it is stochastically ordered. We apply the theory of Lyapunov functions and find explicit estimates for the rate of exponential convergence to the stationary distribution, as time goes to infinity. This continues the work of Lund, Meyn and Tweedie (1996). We apply these results to systems of two competing Levy particles with rank-dependent dynamics.
\end{abstract}

\maketitle

\section{Introduction}

\subsection{Exponential convergence of Markov processes} 
On a state space $\mathfrak X$, consider a Markov process $X = (X(t), t \ge 0)$ with generator $\mathcal M$ and transition kernel $P^t(x, \cdot)$. Existence and uniqueness of a stationary distribution $\pi$ and convergence $X(t) \to \pi$ as $t \to \infty$ have been extensively studied. One common method to prove an exponential rate of convergence to the stationary distribution $\pi$ is to construct a Lyapunov function: that is, a function $V : \mathfrak X \to [1, \infty)$, for which 
$$
\mathcal M V(x) \le -kV(x) + b1_E(x),\ \ x \in \mathfrak X,
$$
where $b, k > 0$ are constants, and $E$ is a ``small'' set. There is a precise term {\it small set} in this theory. In this article, $\mathfrak X = \BR_+ := [0, \infty)$, and for our purposes we can assume $E$ is a compact set. If there exists a Lyapunov function $V$, then (under some additional technical assumptions: irreducibility and aperiodicity), there exists a unique stationary distribution $\pi$, and for every $x \in \mathfrak X$, the transition probability measure $P^t(x, \cdot)$ converges to $\pi$ in total variation as $t \to \infty$. Moreover, the convergence is exponentially fast. More precisely, suppose $\norm{\cdot}$ denotes the total variation norm or a similar norm for signed measures on $\mathfrak X$. (In Section 3, we speciy the exact norm which we are using.) Then for some positive constants $C(x)$ and $\vk$, we have:
\begin{equation}
\label{eq:original-inequality}
\norm{P^t(x, \cdot) - \pi(\cdot)} \le C(x)e^{-\vk t}.
\end{equation}
Results along these lines can be found in \cite{DMT1995, MT1993a, MT1993b}, as well as in many other articles. Similar results are known for discrete-time Markov chains; the reader can consult the classic book \cite{MeynBook}. However, to estimate the constant $\vk$ is a much harder task: See, for example, \cite{BCG2008, Davies1986, MT1994, RR1996, RT1999, Rosenthal1995, Zeifman1991}. In the general case, the exact value of $\vk$ depends in a complicated way on the constants $b$ and $k$, on the set $E$, and on the transition kernel $P^t(x, \cdot)$. 

Under some conditions, however, we can simply take $\vk = k$. This happens when $\mathfrak X = \BR_+$, $E = \{0\}$, and the process $X$ is stochastically ordered. The latter means that if we start two copies $X'$ and $X''$ of this process from initial conditions $x' \le x''$, then we can couple them so that a.s. for all $t \ge 0$, we have: $X'(t) \le X''(t)$. This remarkable result was proved in \cite[Theorem 2.2]{LMT1996}. (A preceding paper \cite{LT1996} contains similar results for stochastically ordered discrete-time Markov chains.)  In addition, in \cite[Theorem 2.4]{LMT1996}, they also prove that even for a possibly non-stochastically ordered Markov process on $\BR_+$, if it is stochastically dominated by another stochastically ordered Markov process on $\BR_+$ with a Lyapunov function with $E = \{0\}$, then the original process converges with exponential rate $\vk = k$. Let us also mention a paper \cite{RT2000}, which generalizes this method for stochastically ordered Markov processes when $E \ne \{0\}$ (however, the results there are not nearly as simple as $\vk = k$). 

%However, at this point, they make a mistake: They take the function $V(x) = e^{\la x}$ for some $\la > 0$. But $V'(0) \ne 0$, and therefore it is not clear whether this function belongs to the domain of the generator $\CL$ of this reflected diffusion. In \cite[Section 6]{LMT1996}, the authors mentioned that they are reducing the setting to the case of a non-reflected diffusion on the real line, for which we do not have this restrictive condition on the domain of the generator. However, their argument does not seem very convincing. 

\subsection{Our results} In this paper, we improve upon these results. First, in Theorem ~\ref{thm:2}, we prove that $\vk = k$ for stochastically ordered processes (a version of \cite[Theorem 2.2]{LMT1996}) under slightly different assumptions, with an improved constant $C(x)$. Second, in Theorem~\ref{thm:non-stoch-ordered}, we prove a stronger version of \cite[Theorem 2.4]{LMT1996} for non-stochastically ordered processes (because the norm in~\eqref{eq:original-inequality} is stronger in our paper). In particular, our result allows for convergence of moments, which does not follow from \cite[Theorem 2.4]{LMT1996}. Third, in Lemma~\ref{lemma:exact-rate}, we show that in a certain case, this rate $\vk$ of convergence is exact: one cannot improve the value of $\vk$; this serves as a counterpart of \cite[Theorem 2.3]{LMT1996}. Next, we apply this theory to reflected jump-diffusions on $\BR_+$.

\subsection{Reflected jump-diffusions}  A reflected jump-diffusion process $Z = (Z(t), t \ge 0)$ on the positive half-line $\BR_+$ is a process that can be described as follows: As long as it is away from zero, it behaves as a diffusion process with drift coefficient $g(\cdot)$ and diffusion coefficient $\si^2(\cdot)$. When it htis zero, it is reflected back into the positive half-line. It can also make jumps: Take a family $(\nu_x)_{x \ge 0}$ of Borel measures on $\BR_+$. If this process is now at point $x \in \BR_+$, it can jump with intensity $r(x) = \nu_x(\BR_+)$, and the destination of this jump is itself a random point, distributed according to the probability measure $r^{-1}(x)\nu_x(\cdot)$. If $r(x) = 0$, then this process cannot jump from $x$. A rigorous definition is given in Section 2.  

We can similarly define reflected diffusions and jump-diffusions in a domain $D \subseteq \BR^d$. These processes have many applications: Among a multitude of existing articles and monographs, let us mention \cite{ChenBook, KellaWhitt1990, KushnerBook, Whitt2001, WhittBook} and references therein for applications to stochastic networks. 
Existence and uniqueness of a stationary distribution and convergence to this stationary distribution as $t \to \infty$ for these processes were intensively studied recently. Among many references, we point out the articles \cite{ABD2001, BudhirajaLee, DW1994} for reflected diffusions and \cite{KellaWhitt, AB2002, Piera2008} for reflected jump-diffusions. However, these papers do not include explicit estimates of the exponential rate of convergence.

In this paper, we first prove a general exponential convergence result for a reflected jump-diffusion on $\BR_+$: this is Theorem~\ref{thm:1}, which does not provide an explicit rate $\vk$ of exponential convergence. Next, we find an explicit rate of convergence for a stochastically ordered reflected jump-diffusion in Theorem~\ref{thm:3}, and for a non-stochastically ordered reflected jump-diffusion (dominated by another stochastically ordered reflected jump-diffusion) in Corollary~\ref{cor:non-stoch-ordered-RJD}. 

%We consider two applications of our results. The first one is a reflected Levy process on the positive half-line. Such process behaves as a usual one-dimensional Levy process as long as it is on $(0, \infty)$. As it hits $0$, it is reflected back to the positive half-line. If it tries to jump over $0$, its jump is redirected to $0$ instead. 

\subsection{Systems of competing Levy particles}
Finally, we apply our results to systems of two {\it competing Levy particles}, which continues the research from \cite{Ichiba11, BFK2005, Shkolnikov2011}. In these systems, each particle is a one-dimensional Levy process. Its drift and diffusion coefficients and the jump measure depend on the current rank of the particle relative to other particles. Such systems are applied in mathematical finance in \cite{CP2010, FernholzKaratzasSurvey, JM2013b}. %In our setting, unlike in the above papers, there can also be dependence between Brownian terms and between jumps of different particles. 

\subsection{Organization of the paper} In Section 2, we introduce all necessary definitions, and construct there reflected jump-diffusion processes. In Section 3, we prove exponential convergence under some fairly general conditions, but without finding or estimating a rate of exponential convergence. In Section 4, we prove $\vk = k$ for stochastically ordered processes, and in Section 5, for processes dominated by a stochastically ordered process. In Section 6, we show that in a certain particular case, our estimate of the rate of convergence is exact. Then we apply these results in Section 7 to systems of two competing Levy particles.

\subsection{Notation} Weak convergence of measures or random variables is denoted by $\Ra$. We denote $\BR_+ := [0, \infty)$ and $\BR_- := (-\infty, 0]$. The Dirac point mass at the point $x$ is denoted by $\de_x$. Take a Borel (signed) measure $\nu$ on $\BR_+$. For a function $f : \BR_+ \to \BR$, denote $(\nu, f) := \int_{\BR_+}f\md\nu$. For a function $f : \BR_+ \to [1, \infty)$, we define the following norm:
\begin{equation}
\label{eq:f-norm}
\norm{\nu}_f := \sup\limits_{|g| \le f}|(\nu, g)|.
\end{equation}
For $f \equiv 1$, this is the {\it total variation norm}: $\norm{\cdot}_f \equiv \norm{\cdot}_{\TV}$. In the rest of the article, we operate on a filtered probability space $(\Oa, \CF, (\CF_t)_{t \ge 0}, \MP)$ with the filtration satisfying the usual conditions. For a function $f : \BR_+ \to \BR$, we let $\norm{f} := \sup_{x \ge 0}|f(x)|$. We denote by $\Exp(\la)$ the exponential distribution on the positive half-line with mean $\la^{-1}$ (and rate $\la$). 

\section{Definition and Construction of Reflected Jump-Diffusions} First, let us define a reflected diffusion on $\BR_+$ without jumps. Take functions $g : \BR_+ \to \BR$ and $\si : \BR_+ \to (0, \infty)$. Consider an $(\CF_t)_{t \ge 0}$\,-\,Brownian motion $W = (W(t), t \ge 0)$. 

\begin{defn} A continuous adapted $\BR_+$-valued process $Y = (Y(t), t \ge 0)$ is called a {\it reflected diffusion} on $\BR_+$ with {\it drift coefficient} $g(\cdot)$ and {\it diffusion coefficient} $\si^2(\cdot)$, {\it starting from} $y_0 \in \BR_+$, if there exists another real-valued continuous nondecreasing adapted process $\ml = (\ml(t), t \ge 0)$ with $\ml(0) = 0$, which can increase only when $Y = 0$, such that for $t \ge 0$ we have:
$$
Y(t) = y_0 + \int_0^tg(Y(s))\md s + \int_0^t\si(Y(s))\md W(s) + \ml(t).
$$
\label{defn:SDE-R}
\end{defn}

\begin{asmp} The functions $g$ and $\si$ are Lipschitz continuous: for some constant $C_L > 0$,
$$
|g(x) - g(y)| + |\si(x) - \si(y)| \le C_L|x - y|,\ \ \mbox{for all}\ \ x, y \in \BR_+.
$$
Moreover, the function $\si$ is bounded away from zero: $\inf_{x \ge 0}\si(x) > 0$. 
\end{asmp}

It is well known (see, for example, the classic paper \cite{Skorohod}) that under Assumption 1, for every $y_0 \in \BR_+$, there exists a weak version of the reflected diffusion from Definition~\ref{defn:SDE-R}, starting from $y_0$, which is unique in law. Moreover, for different starting points $y_0 \in \BR_+$, these processes form a Feller continuous strong Markov family. We can define the transition semigroup $P^t$ which acts on functions: $f \mapsto P^tf$, as well as the transition kernel $P^t(x, \cdot)$ and the generator $\CA$:
\begin{equation}
\label{eq:CA}
\CA f(x) = g(x)f'(x) + \frac12\si^2(x)f''(x),\ \ \mbox{if}\ \ f'(0) = 0.
\end{equation}
Take a family $(\nu_x)_{x \ge 0}$ of finite Borel measures $\nu_x$ on 
$\BR_+$.

\begin{asmp} The family $(\nu_x)_{x \ge 0}$ is weakly continuous: $\nu_{x_n} \Ra \nu_{x_0}$ for $x_n \to x_0$. In addition, the function $r(x) := \nu_x(\BR_+)$ is bounded on $\BR_+$: $\sup_{x \ge 0}r(x) =: \rho < \infty$. 
\end{asmp}

Take a reflected diffusion $Y = (Y(t), t \ge 0)$ on $\BR_+$ with drift coefficient $g$ and diffusion coefficient $\si^2$. Using this process $Y$, let us construct a weak version of the reflected jump-diffusion $Z = (Z(t), t \ge 0)$ with the same drift and diffusion coefficients and with the family $(\nu_x)_{x \in \BR_+}$ of jump measures, starting from $y_0 \in \BR_+$. One way to do this is {\it piecing out}. 

For every $y \in \BR_+$, take infinitely many i.i.d. copies $Y^{(y, n)}$, $n = 1, 2, \ldots$ of the reflected diffusion $Y$, starting from $Y^{(y, n)}(0) = y$. For every $x \in \BR_+$, generate infinitely many i.i.d. copies $\xi^{(x, n)}$ of a random variable $\xi^{(x)} \sim r^{-1}(x)\nu_x(\cdot)$, independent of each other and of the copies of the reflected  diffusion $Y$. We assume all processes $Y^{(y, n)}$ are adapted to $(\CF_t)_{t \ge 0}$, and every $\sigma$-subalgebra $\CF_t$ contains all $\xi^{(x, n)}$ for $x \in \BR_+$ and $n = 1, 2, \ldots$ Start a process $Y^{(y_0, 1)}$. We kill it with intensity $r(Y^{(y_0, 1)}(t))$: If $\zeta_1$ is the killing time, then 
\begin{equation}
\label{eq:stopping-time}
\MP\left(\zeta_1 > t\right) = \exp\left(-\int_0^tr\left(Y^{(y_0, 1)}(s)\right)\md s\right).
\end{equation}
If $\zeta_1 < \infty$, let $x_1 := Y^{(y_0, 1)}(\zeta_1)$, and let $y_1 := \xi^{(x_1, 1)}$. Start the process $Y^{(y_1, 2)}$, and kill it at time $\zeta_2$ with intensity $r(Y^{(y_1, 2)}(t))$, similarly to~\eqref{eq:stopping-time}, etc. Because $r(x) \le \rho$ for $x \in \BR_+$, one can find a sequence of i.i.d. $\eta_1, \eta_2, \ldots \sim \Exp(\rho)$ such that a.s. for all $k$ we have: $\zeta_k \ge \eta_k$. Therefore,
\begin{equation}
\label{eq:tends-to-infty}
\tau_k := \zeta_1 + \ldots + \zeta_k \ge \eta_1 + \ldots + \eta_k \to \infty\ \ \mbox{a.s. as}\ \ k \to \infty.
\end{equation}
Define the process $Z = (Z(t), t \ge 0)$ as follows: for $t \in [\tau_k, \tau_{k+1})$, $k = 0, 1, 2, \ldots$, let $Z(t) = Y^{(y_k, k+1)}(t - \tau_k)$.
In other words, it jumps at moments $\tau_1, \tau_2, \ldots$, and behaves as a reflected diffusion without jumps on $\BR_+$ on each interval $(\tau_k, \tau_{k+1})$. Because of~\eqref{eq:tends-to-infty}, this defines $Z(t)$ for all $t \ge 0$. The following result is proved in \cite[Theorem 2.4, Theorem 5.3, Example 1]{Sawyer}. 

\begin{prop} Under Assumptions 1 and 2, the construction above yields a Feller continuous strong Markov process on $\BR_+$ with generator 
\begin{equation}
\label{eq:CL}
\CL = \CA + \CN,
\end{equation}
where the operator $\CA$ is given by~\eqref{eq:CA}, and $\CN$ is defined by 
\begin{equation}
\label{eq:CN}
\CN f(x) := \int_0^{\infty}\left[f(y) - f(x)\right]\nu_x(\md y).
\end{equation}
%This generator $\CL f$ is well-defined for all $f \in \CD(\CA)$ such that $\CN f(x) < \infty$ for all $x \in \BR_+$. 
\label{prop:Feller}
\end{prop}

\section{Lyapunov Functions and Exponential Convergence} 

In this section, we define Lyapunov functions of a Markov process on $\BR_+$ and relate them to convergence of this Markov process to its stationary distribution with an exponential rate. We apply this theory to the case of reflected jump-diffusions. However, we do not find an explicit rate $\vk$ of exponential convergence: this requires stochastic ordering, which is done in Section 4. Our definitions are taken from \cite{MyOwn10} and are slightly different from the usual definitions in the classic articles \cite{DMT1995, MT1993a, MT1993b}. These adjusted definitions seem to be more suitable for our purposes. 

\subsection{Notation and definitions} Take a Feller continuous strong Markov family $(P^t)_{t \ge 0}$ on $\BR_+$ with generator $\mathcal M$, which has a domain $\CD(\mathcal M)$. Let $P^t(x, \cdot)$ be the corresponding transition kernel. Slightly abusing the terminology, we will sometimes speak interchangeably about the Markov process $X = (X(t), t \ge 0)$ or the Markov kernel $(P^t)_{t \ge 0}$. We use the standard Markovian notation: $\mu P^t$ is the result of the action of $P^t$ on a measure $\mu$; symbols $\MP_x$ and $\ME_x$ correspond to the copy of $X$ starting from $X(0) = x$. 

%\begin{defn} A probability distribution $\pi$ on $\BR_+$ is called {\it stationary} if from $Z(0) \sim \pi$ it follows that $Z(t) \sim \pi$. %A stationary distribution $\pi$ is called {\it limiting} if for every initial distribution $Z(0)$, we have: $Z(t) \Ra \pi$ as $t \to \infty$. 
%\end{defn}

%\begin{rmk} It is easy to see that if a stationary distribution is limiting, then it is unique.
%\end{rmk}

\begin{defn} Take a continuous function $V : \BR_+ \to [1, \infty)$ in the domain 
$\CD(\mathcal M)$. Assume there exist constants $b, k, z > 0$ such that 
\begin{equation}
\label{eq:Lyapunov}
\mathcal M V(x) \le - kV(x) + b1_{[0, z]}(x),\ \ x \in \BR_+.
\end{equation}
Then $V$ is called a {\it Lyapunov function} for this Markov family with {\it Lyapunov constant} $k$.  
\label{defn:Lyapunov}
\end{defn}

Let us now define the concept of exponential convergence to the stationary distribution. 
Take a function $W : \BR_+ \to [1, \infty)$ and a constant $\vk > 0$. Recall the definition of the norm $\norm{\cdot}_W$ from~\eqref{eq:f-norm}.

%\begin{defn} This process is called {\it ergodic} if there exists a unique stationary distribution $\pi$, and for every $x \in \BR_+$, we have:
%$$
%\norm{P^t(x, \cdot) - \pi(\cdot)}_{\TV} \to 0\ \ \mbox{as}\ \ t \to \infty.
%$$
%\end{defn}

\begin{defn} %This process is called {\it exponentially ergodic} with an {\it exponential rate of convergence} $\vk$ if there exists a unique stationary distribution $\pi$, and
%\begin{equation}
%\label{eq:exp-ergodicity}
%\norm{P^t(x, \cdot) - \pi(\cdot)}_{\TV} \le W(x)e^{-\vk t},\ \ x \in \BR_+,\ \ t \ge 0.
%\end{equation}
This process is called {\it $W$-uniformly ergodic} with an {\it exponential rate of convergence} $\vk$ if there exists a unique stationary distribution $\pi$, and for some constant $D$, we have:
\begin{equation}
\label{eq:uniform-ergodicity}
\norm{P^t(x, \cdot) - \pi(\cdot)}_W \le DW(x)e^{-\vk t},\ \ x \in \BR_+,\ \ t \ge 0.
\end{equation}
\label{defn:exp-conv}
\end{defn}

Finally, let us introduce a technical property of this Markov family, which allows us to link Lyapunov functions from Definition~\ref{defn:Lyapunov} with exponential convergence from Definition~\ref{defn:exp-conv}. 

\begin{defn} The Markov process is called {\it totally irreducible} if for every $t > 0$, $x \in \BR_+$, and a subset $A \subseteq \BR_+$ of positive Lebesgue measure, we have: $P^t(x, A) > 0$. 
\end{defn}

%\begin{lemma} Under Assumptions 1 and 2, the reflected jump-diffusion constructed in Section 2 is totally irreducible. 
%\label{lemma:ti}
%\end{lemma}
%
%\begin{proof}
%The corresponding reflected diffusion without jumps is totally irreducible: For $\si$ bounded away from zero, this was proved in \cite[Lemma 5.7]{BudhirajaLee}, and the general case can be easily reduced to this particular one. Indeed, assume the starting point is $x$. Take a $y > x$ and consider another reflected diffusion $\ol{Z} = (\ol{Z}(t), t \ge 0)$ on $\BR_+$, starting from $\ol{Z}(0) = x$, with drift and diffusion coefficients
%$$
%\ol{g}(z) := g(z\wedge y),\ \ \ol{\si}^2(z) := \si^2(z\wedge y),\ \ z \in \BR_+.
%$$
%Then the reflected diffusion $\ol{Z}$ behaves as $Z$ until it hits $y$. Therefore,
%$$
%\MP(Z(t) \in A) \ge \MP\left(\ol{Z}(t) \in A,\ \max\limits_{0 \le s \le t}\ol{Z}(s) < y\right) \ge \MP(\ol{Z}(t) \in A) - \MP\left(\max\limits_{0 \le s \le t}\ol{Z}(s) \ge y\right). 
%$$
%\end{proof}

The following result is the connection between Lyapunov functions and exponential convergence. It was proved in \cite{MyOwn10} and is slightly different from classic
results of \cite{MT1993a, MT1993b, DMT1995}.

\begin{prop} Assume the Markov process is totally irreducible with a Lyapunov function $V$. Then the process is $V$-uniformly ergodic, and the stationary distribution $\pi$ satisfies $(\pi, V) < \infty$. 
\label{prop:fund}
\end{prop} 

\subsection{Main results} Let us actually construct a Lyapunov function for the reflected jump-diffusion process $Z = (Z(t), t \ge 0)$ from Section 2. We would like to take the following function:
\begin{equation}
\label{eq:V}
V_{\la}(x) := e^{\la x},\ \ x \in \BR_+,
\end{equation}
for some $\la > 0$. Indeed, the first and second derivative operators from~\eqref{eq:CA}, included in the generator $\CL$ from~\eqref{eq:CL}, act on this function in a simple way. However, $V'_{\la}(0) \ne 0$, which contradicts~\eqref{eq:CA}. Therefore, we cannot simply take $V_{\la}$ as a Lyapunov function; we need to modify it. Fix $s_2 > s_1 > 0$ and take a nondecreasing $C^{\infty}$ function $\phi : \BR_+ \to \BR_+$ such that 
\begin{equation}
\label{eq:phi}
\phi(s) = 
\begin{cases}
0,\ s \le s_1;\\
s,\ s \ge s_2;
\end{cases}
\ \ \phi(s) \le s\ \ \mbox{for}\ \ s \ge 0.
\end{equation}
The easiest way to construct this function is as follows. Take a mollifier: a nonnegative $C^{\infty}$ function $\oa : \BR \to \BR_+$ with $\int_{\BR}\oa(x)\md x = 1$, with support $[-\eps, \eps]$, where $\eps = (s_2 - s_1)/3$. One example of this is
$$
\oa_{\eps}(x) = c\exp\left(-(\eps - |x|)^{-1}\right),\ \ c > 0.
$$
Apply this mollifier to the following piecewise linear function:
$$
\ol{\phi}(s) = 
\begin{cases}
0,\ \ s \le s_1 + \eps;\\
(s_2-\eps)\frac{s-s_1 - \eps}{s_2-s_1 - 2\eps},\ \ s_1 +\eps \le s \le s_2 - \eps;\\
s,\ \ s \ge s_2 - \eps.
\end{cases}
$$
The convolution of $\ol{\phi}$ and $\omega_{\eps}$ gives us this necessary function $\phi$ which satisfies~\eqref{eq:phi}. Define a new candidate for a Lyapunov function: 
\begin{equation}
\label{eq:ol-V}
\ol{V}_{\la}(x) := e^{\la\phi(x)},\ \ x \in \BR_+.
\end{equation}
This function satisfies $\ol{V}'_{\la}(0) = 0$, because $\phi'(0) = 0$. Let us impose an additional assumption on the family $(\nu_x)_{x \in \BR_+}$ of jump measures.

\begin{asmp}
There exists a $\la_0 > 0$ such that 
$$
\sup\limits_{x \ge 0}\int_0^{\infty}e^{\la_0|y - x|}\nu_x(\md y) < \infty.
$$
\end{asmp}

Under Assumption 3, we can define the following quantity for $\la \in [0, \la_0]$ and $x \in \BR_+$:
\begin{equation}
\label{eq:k}
K(x, \la) = g(x)\la + \frac12\si^2(x)\la^2 + \int_0^{\infty}\left[e^{\la(y - x)} - 1\right]\nu_x(\md y).
\end{equation}
Now comes one of the two main results of this paper. This first result is a statement of convergence with exponential rate, but it does not provide an estimate for this rate.  

\begin{thm} Under Assumptions 1, 2, 3, suppose there exists a $\la \in (0, \la_0)$ such that
\begin{equation}
\label{eq:limsup-k}
\varlimsup\limits_{x \to \infty}K(x, \la) < 0.
\end{equation}
Then the reflected jump-diffusion is $V_{\la}$-uniformly ergodic, and the (unique) stationary distribution $\pi$ satisfies $(\pi, V_{\la}) < \infty$. 
\label{thm:1}
\end{thm} 

\begin{proof} Note that $V_{\la}$-uniform ergodicity and $\ol{V}_{\la}$-uniform ergodicity are equivalent, because these functions are themselves equivalent in the following sense: there exist constants $c_1, c_2$ such that 
$$
0 < c_1 \le \frac{\ol{V}_{\la}(x)}{V_{\la}(x)} \le c_2 < \infty\ \ \mbox{for all}\ \ x \in \BR_+.
$$
The corresponding reflected diffusion without jumps is totally irreducible, see \cite[Lemma 5.7]{BudhirajaLee}. The reflected jump-diffusion is also totally irreducible: With probability at least $e^{-\rho t}$ there are no jumps until time $t > 0$, and the reflected jump-diffusion behaves as a reflected diffusion without jumps.
Therefore, by Proposition~\ref{prop:fund} it is sufficient to show that $\ol{V}_{\la}$ is a Lyapunov function (in the sense of Definition~\ref{defn:Lyapunov}) for this reflected jump-diffusion. Apply the generator $\CL = \CA + \CN$ from~\eqref{eq:CL} to the function $\ol{V}_{\la}$. For $x > s_2$, we have: $\ol{V}_{\la}(x) = V_{\la}(x)$. Now, the operator $\CA$ from~\eqref{eq:CA} is a differential operator, and its value at the point $x$ depends on its value in an arbitrarily small neighborhood of $x$. Therefore, for $x > s_2$, 
\begin{equation}
\label{eq:CA-V}
\CA\ol{V}_{\la}(x) = \CA V_{\la}(x) = \left[g(x)\la + \frac12\si^2(x)\la^2\right]V_{\la}(x).
\end{equation}
Apply the operator $\CN$ from~\eqref{eq:CN} to the function $\ol{V}_{\la}$. For $y \in \BR_+$, we have: $\phi(y) \le y$, and therefore 
\begin{equation}
\label{eq:phi-vs-no-phi}
\ol{V}_{\la}(y)  = e^{\la\phi(y)} \le V_{\la}(y) = e^{\la y}.
\end{equation}
From~\eqref{eq:phi-vs-no-phi}, we get the following comparison: for $x \ge s_2$ and $y \in \BR_+$, 
\begin{equation}
\label{eq:V-vs-ol-V}
\ol{V}_{\la}(y) - \ol{V}_{\la}(x)  \le e^{\la y} - e^{\la x} = e^{\la x}\left(e^{\la(y - x)} - 1\right) = \ol{V}_{\la}(x)\left(e^{\la(y - x)} - 1\right).
\end{equation}
Because of~\eqref{eq:V-vs-ol-V}, we have the following estimate for the operator $\CN$:  
\begin{equation}
\label{eq:CN-V}
\CN\ol{V}_{\la}(x) \le \ol{V}_{\la}(x)\int_0^{\infty}\left[e^{\la(y-x)} - 1\right]\nu_x(\md y).
\end{equation}
Combining~\eqref{eq:CA-V} and~\eqref{eq:CN-V}, and recalling the definition of $k(x, \la)$ from~\eqref{eq:k}, we get:
\begin{equation}
\label{eq:CL-V}
\CL\ol{V}_{\la}(x) = \CA\ol{V}_{\la}(x) + \CN\ol{V}_{\la}(x) \le K(x, \la)\ol{V}_{\la}(x)\ \ \mbox{for}\ \ x \ge s_2,\ \la \in (0, \la_0).
\end{equation}
Let us state separately the following technical lemma. 

\begin{lemma} For every $\la \in (0, \la_0)$, the function  $\CL\ol{V}_{\la}(x)$ is bounded with respect to $x$ on $[0, s_2]$. 
\label{lemma:aux-cont-0}
\end{lemma}

Assuming we already proved Lemma~\ref{lemma:aux-cont-0}, let us complete the proof of Theorem~\ref{thm:1}. Denote 
\begin{equation}
\label{eq:c-0}
c_0 := \sup\limits_{x \in [0, s_2]}\left[\CL\ol{V}_{\la}(x) + k_0\ol{V}_{\la}(x)\right] < \infty.
\end{equation}
Combining~\eqref{eq:c-0} with~\eqref{eq:CL-V}, we get that 
$$
\CL\ol{V}_{\la}(x) \le -k_0\ol{V}_{\la}(x) + c_01_{[0, s_2]}(x),\ \ x \in \BR_+,
$$
which completes the proof of Theorem~\ref{thm:1}. 

\smallskip

{\it Proof of Lemma~\ref{lemma:aux-cont-0}.} The function $\ol{V}_{\la}$ has continuous first and second derivatives, and by Assumption 1 the functions $g$ and $\si^2$ are also continuous. Therefore, $\CA\ol{V}_{\la}$ is bounded on $[0, s_2]$. Next, let us show that the following function is bounded on $[0, s_2]$:
\begin{equation}
\label{eq:integral}
\CN\ol{V}_{\la}(x) \equiv \int_0^{\infty}\left[\ol{V}_{\la}(y) - \ol{V}_{\la}(x)\right]\nu_x(\md y) = \int_0^{\infty}\ol{V}_{\la}(y)\nu_x(\md y) - \ol{V}_{\la}(x)r(x).
\end{equation}
The function~\eqref{eq:integral} can be estimated from below by 
$-\ol{V}_{\la}(x)r(x)$, which is continuous and therefore bounded on $[0, s_2]$. On the other hand,~\eqref{eq:integral} can be estimated from above by 
\begin{equation}
\label{eq:new-int}
\int_0^{\infty}\ol{V}_{\la}(y)\nu_x(\md y) \le \int_0^{\infty}V_{\la}(y)\nu_x(\md y) = e^{\la x}\int_0^{\infty}e^{\la|y-x|}\nu_x(\md y).
\end{equation}
From Assumption 3, it is easy to get that the function from~\eqref{eq:new-int} is bounded on $[0, s_2]$. This completes the proof of Lemma~\ref{lemma:aux-cont-0}, and with it the proof of Theorem~\ref{thm:1}. 
\end{proof}

Now, let us find some examples. Define the function
$$
m(x) := g(x) + \int_0^{\infty}(y-x)\nu_x(\md y),\ \ x \in \BR_+.
$$
This is a ``joint drift coefficient'' at the point $x \in \BR_+$, which combines the actual drift coefficien $g(x)$ and the average displacement $y - x$ for the jump from $x$ to $y$, where $y \sim [r(x)]^{-1}\nu_x(\cdot)$. One can assume that if $m(x) < 0$ for all or at least for large enough $x \in \BR_+$, then the process has a unique stationary distribution. This is actually true, with some qualifications. 

\begin{cor} Under Assumptions 1, 2, 3, suppose $\si^2$ is bounded on $\BR_+$, and 
\begin{equation}
\label{eq:m<0}
\varlimsup\limits_{x \to \infty}m(x) < 0.
\end{equation}
Then there exists a $\la \in (0, \la_0)$ such that~\eqref{eq:limsup-k} holds. By Theorem~\ref{thm:1}, the reflected jump-diffusion $Z = (Z(t), t \ge 0)$ is $V_{\la}$-uniformly ergodic, and the stationary distribution $\pi$ satisfies $(\pi, V_{\la}) < \infty$. 
\label{cor:m<0}
\end{cor} 

\begin{proof}
Because of Assumption 3, we can take the first and second derivative with respect to $\la \in (0, \la_0)$ inside the integral in~\eqref{eq:k}. Therefore,
\begin{equation}
\label{eq:dk}
\frac{\pa K}{\pa\la} = g(x) + \si^2(x)\la + \int_0^{\infty}e^{\la(y-x)}(y - x)\nu_x(\md y),
\end{equation}
\begin{equation}
\label{eq:d2k}
\frac{\pa^2K}{\pa\la^2} = \si^2(x) + \int_0^{\infty}e^{\la(y-x)}(y - x)^2\nu_x(\md y).
\end{equation}
Letting $\la = 0$ in~\eqref{eq:dk}, we have:
\begin{equation}
\label{eq:dk-m}
\left.\frac{\pa K}{\pa \la}\right|_{\la = 0} = g(x) + \int_0^{\infty}(y-x)\nu_x(\md y) = m(x).
\end{equation}
By the condition~\eqref{eq:m<0}, there exist $m_0 > 0$ and $x_0 > 0$ such that 
\begin{equation}
\label{eq:m<m0}
m(x) \le  -m_0\ \ \mbox{for}\ \ x \ge x_0.
\end{equation}
There exists a constant $C_1 > 0$ such that for all $z \ge 0$, we have: $z^2 \le C_1e^{\la_0 z/2}$. Applying this to $z = |y - x|$, we have: for $\la \in [0, \la_0/2]$, 
\begin{equation}
\label{eq:comparison}
\int_0^{\infty}(y-x)^2e^{\la(y-x)}\nu_x(\md y) \le C_1\int_0^{\infty}e^{\la_0|y-x|}\nu_x(\md y).
\end{equation}
Combining~\eqref{eq:comparison} with Assumption 3 and the boundedness of $\si^2$, we get that the right-hand side of~\eqref{eq:d2k} is bounded for $\la \in [0, \la_0/2]$ and $x \in \BR_+$. Let $C_2$ be this bound:
\begin{equation}
\label{eq:C-2}
C_2 := \sup\limits_{\substack{x \in \BR_+\\ \la \in (0, \la_0/2]}}\left|\frac{\pa^2K}{\pa\la^2}(x, \la)\right|.
\end{equation}
By Taylor's formula, for some $\tilde{\la}(x) \in [0, \la]$, we have:
\begin{equation}
\label{eq:Taylor}
K(x, \la) = K(x, 0) + \la\frac{\pa K}{\pa\la}(x, 0) + \frac{\la^2}2\frac{\pa^2K}{\pa\la^2}(x, \tilde{\la}(x)).
\end{equation}
Plugging~\eqref{eq:dk-m} and $K(x, 0) = 0$ into~\eqref{eq:Taylor} and using the estimate~\eqref{eq:C-2}, we have:
\begin{equation}
\label{eq:k-estimate}
K(x, \la) \le \la m(x) + \frac{C_2}2\la^2,\ \la \in \left[0, \frac{\la_0}2\right].
\end{equation}
Combining~\eqref{eq:k-estimate} with~\eqref{eq:m<m0}, we get: 
$$
K(x, \la) \le \tilde{K}(\la) := -\la m_0 + \frac{C_2}2\la^2,\ \ x \ge x_0,\ \ \la \in \left[0, \frac{\la_0}2\right].
$$
It is easy to see that $\tilde{K}(\la) < 0$ for $\la \in (0, 2m_0/C_2]$. Letting 
$$
\la = \frac{2m_0}{C_2}\wedge\frac{\la_0}2,
$$
we complete the proof of Corollary~\ref{cor:m<0}. 
\end{proof}

\begin{rmk} In the setting of Theorem~\ref{thm:1} or Corollary~\ref{cor:m<0}, the convergence of moments of $P^t(x, \cdot)$ to the moments of $\pi(\cdot)$ follows from $V_{\la}$-uniform ergodicity. Indeed, take an $\al > 0$. There exists a constant $C(\al, \la) > 0$ such that $x^{\al} \le C(\al, \la)V_{\la}(x)$ for $x \in \BR_+$. From~\eqref{eq:uniform-ergodicity} we get: for $x \in \BR_+, t \ge 0$, 
$$
\left|\int_0^{\infty}y^{\al}P^t(x, \md y) - \int_0^{\infty}y^{\al}\pi(\md y)\right| \le C(\al, \la)KV_{\la}(x)e^{-\vk t}.
$$
\label{rmk:moments}
\end{rmk}

\section{Stochastic Ordering and Explicit Rates of Exponential Convergence}

In this section, we get to the main goal of this article: to explicitly estimate $\vk$, the rate of exponential convergence from~\eqref{eq:uniform-ergodicity}. In case when the reflected jump-diffusion is {\it stochastically ordered}, and $z = 0$ in~\eqref{eq:Lyapunov}, we can (under some additional technical assumptions) get that $\vk = k$. 

\subsection{General results for Markov processes} For two finite Borel measures $\nu$ and $\nu'$ on $\BR_+$ with $\nu(\BR_+) = \nu'(\BR_+)$, we write $\nu \preceq \nu'$, or $\nu' \succeq \nu$, and say that $\nu$ {\it is stochastically dominated by} $\nu'$, if for every $z \in \BR_+$, we have: $\nu([z, \infty)) \le \nu'([z, \infty))$. 

\begin{defn} A family $(\nu_x)_{x \ge 0}$ of finite Borel measures, with $\nu_x(\BR_+)$ independent of $x$, is called {\it stochastically ordered} if $\nu_x(\BR_+)$ does not depend on $x$, and $\nu_{x} \preceq \nu_{y}$ for $x \le y$. A Markov transition kernel $P^t(x, \cdot)$, or, equivalently, the corresponding Markov process is called {\it stochastically ordered}, if for every $t > 0$, the family $(P^t(x, \cdot))_{x \ge 0}$ is stochastically ordered.
\label{defn:stoch-ordered}
\end{defn} 

\begin{rmk} An equivalent definition of a Markov process $X = (X(t), t \ge 0)$ on $\BR_+$ being stochastically ordered is when we can couple two copies of this process starting from different initial points such that they can be compared pathwise. More precisely, for all $x, y$ such that $0 \le x \le y$, we can find a probability space with two copies $X^{(x)}$ and $X^{(y)}$ starting from $X^{(x)}(0) = x$ and $X^{(y)}(0) = y$ respectively, and $X^{(x)}(t) \le X^{(y)}(t)$ a.s. for all $t \ge 0$; this follows from \cite{Kamae}. 
\label{rmk:Kamae}
\end{rmk}

In this section, we would also like to make $V_{\la}$ from~\eqref{eq:V} a Lyapunov function as in~\eqref{eq:Lyapunov} with $z = 0$. However, we cannot directly apply the generator $\CL$ from~\eqref{eq:CL} to this function, for the reason we already mentioned: $V'_{\la}(0) \ne 0$, which contradicts~\eqref{eq:CA}. Neither can we use the function $\ol{V}_{\la}$ from~\eqref{eq:ol-V}: as follows from the proof of Theorem~\ref{thm:1}, we would have $z = s_2 > 0$, where $s_2$ is taken from~\eqref{eq:phi}. In \cite{LMT1996}, this obstacle is bypassed by switching to a (non-reflected) diffusion on the whole real line, but we resolve this difficulty in a slightly different way. The proofs of \cite[Theorem 2.1, Theorem 2.2]{LMT1996}, mainly use the Lyapunov condition~\eqref{eq:Lyapunov} only ``until the hitting time of $0$''. To formalize this, let us adjust Definition~\ref{defn:Lyapunov}. Let $\tau(0) := \inf\{t \ge 0\mid X(t) = 0\}$.

\begin{defn} A function $V : \BR_+ \to [1, \infty)$ is called a {\it modified Lyapunov function} with a {\it Lyapunov constant} $k > 0$ if the following process is a supermartingale for every starting point $X(0) = x \in \BR_+$:
\begin{equation}
\label{eq:M-mart}
M(t) := V(X(t\wedge\tau(0))) + k\int_0^{t\wedge\tau(0)}V(X(s))\md s,\ \ t \ge 0.
\end{equation}
\label{defn:Lyapunov-modified}
\end{defn}

\begin{rmk} It is straightforward to prove that if $V$ is a Lyapunov function from Definition~\ref{defn:Lyapunov} with Lyapunov constant $k$ and with $z = 0$ from~\eqref{eq:Lyapunov}, then $V$ is a modified Lyapunov function with Lyapunov constant $k$. In other words, Definition~\ref{defn:Lyapunov-modified} is a generalization of Definition~\ref{defn:Lyapunov} with $z = 0$. Indeed, because $\CM$ is the generator of $X$, the following process is a local martingale:
$$
V(X(t)) - \int_0^t\CM V(X(s))\md s,\ \ t \ge 0.
$$
If $V$ is a Lyapunov function from Definition~\ref{defn:Lyapunov} with Lyapunov constant $k$ and with $z = 0$, then the following process is a local supermartingale:
$$
\bar{M}(t) := V(X(t)) - \int_0^t\left[-kV(X(s)) + b1_{\{0\}}(X(s))\right]\md s,\ \ t \ge 0.
$$
Moreover, this is an actual supermartingale, because it is bounded from below (use Fatou's lemma). Therefore, the process $(\bar{M}(t\wedge\tau(0)), t \ge 0)$ is also a supermartingale by the optional stopping theorem. It suffices to note that $\bar{M}(t\wedge\tau(0)) \equiv M(t)$.  
\end{rmk}

%\begin{asmp}
%For all $0 < x < y$ and $t > 0$, we have:
%$$
%\MP_x\left(\tau(0) > t,\ Z(t) \ge y\right) > 0.
%$$
%\end{asmp}
%
%\begin{rmk} For a reflected jump-diffusion $Z$ constructed in Section 2, Assumption 4 follows from Assumptions 1 and 2. Indeed, with probability greater than or equal to $e^{-\rho t}$, there are no jumps until time $t$, and the process behaves as a reflected diffusion without jumps on the time interval $[0, t]$. Let $y' := 2y - x > y$, and let $\eps := x\wedge(y-x) > 0$. Then
%$$
%\Oa_1 := \{\tau(0) > t,\ Z(t) \ge y\} \supseteq \Oa_2 := \left\{\sup\limits_{0 \le s \le t}\left|Z(t) - x - (y' - x)\frac{s}{t}\right| < \eps\right\}.
%$$
%It follows from \cite[Exercise 6.7.5]{StroockVaradhanBook} (see also  \cite{StroockVaradhan1970}) that $\MP(\Oa_2) > 0$. Thus,  $\MP(\Oa_1) > 0$. 
%\label{rmk:asmp4}
%\end{rmk}

The following is an adjusted version of \cite[Theorem 2.2]{LMT1996}, which states that for the case of a stochastically ordered Markov process with $z = 0$ in~\eqref{eq:Lyapunov}, we can take $\vk = k$ in~\eqref{eq:uniform-ergodicity}. 
Note that we do not require condition (2.1) from \cite{LMT1996}, but we require $(\pi, V) < \infty$ instead. For reflected jump-diffusions, this assumption $(\pi, V) < \infty$ can be obtained from Theorem~\ref{thm:1}, which does not state the exact rate $\vk$ of exponential convergence.

\begin{thm} Suppose $X = (X(t), t \ge 0)$ is a stochastically ordered Markov process on $\BR_+$. Assume there exists a nondecreasing modified Lyapunov function $V$ with a Lyapunov constant $k$.

\medskip

(a) Then for every $x_1, x_2 \in \BR_+$, we have: 
\begin{equation}
\label{eq:difference-pt}
\norm{P^t(x_1, \cdot) - P^t(x_2, \cdot)}_V \le \left[V(x_1) + V(x_2)\right]e^{-kt},\ \ t \ge 0;
\end{equation}

\medskip

(b) For initial distributions $\mu_1$ and $\mu_2$ on $\BR_+$, with $(\mu_1, V) < \infty$ and $(\mu_2, V) < \infty$, we have:
\begin{equation}
\label{eq:difference-distributions-pt}
\norm{\mu_1P^t - \mu_2P^t}_V \le \left[(\mu_1, V) + (\mu_2, V)\right]e^{-kt},\ \ t \ge 0;
\end{equation}

\medskip

(c) If the process $X$ has a stationary distribution $\pi$ which satisfies $(\pi, V) < \infty$, then this stationary distribution is unique, and the process $X$ is $V$-uniformly ergodic with exponential rate of convergence $\vk = k$. More precisely, we have the following estimate:
\begin{equation}
\label{eq:difference-distributions-pt}
\norm{P^t(x, \cdot) - \pi}_V \le \left[(\pi, V) + V(x)\right]e^{-kt},\ \ t \ge 0;
\end{equation}
\label{thm:2}
\end{thm}

Theorem~\ref{thm:2} is an immediate corollary of Theorem~\ref{thm:non-stoch-ordered}. 

\subsection{Application to reflected jump-diffusions} To apply Theorem~\ref{thm:2} to reflected jump-diffusions, let us find when a reflected jump-diffusion on $\BR_+$ is stochastically ordered.

\begin{lemma} Assume the family $(\nu_x)_{x \in \BR_+}$ is stochastically ordered. Then the reflected jump-diffusion from Section 2 is also stochastically ordered. 
\label{lemma:stoch-RJD}
\end{lemma}

\begin{proof} This statement is well known; however, for the sake of completeness, let us present the proof. Let $y \ge x \ge 0$. Following Remark~\ref{rmk:Kamae}, let us construct two copies $Z^{(x)}$ and $Z^{(y)}$ of the reflected jump-diffusion, starting from $Z^{(x)}(0) = x$ and $Z^{(y)}(0) = y$, such that a.s. for $t \ge 0$ we have: $Z^{(x)}(t) \le Z^{(y)}(t)$. Because $\nu_x(\BR_+) = r(x) = r$ does not depend on $x \in \BR_+$ (this follows from Definition~\ref{defn:stoch-ordered}), we can assume the jumps of these two processes happen at the same times, and these jumps $\tau_1 \le \tau_2 \le \ldots$ form a Poisson process on $\BR_+$ with rate $r$. That is, $\tau_n - \tau_{n-1}$ are i.i.d. $\Exp(r)$; for consistency of notation, we let $\tau_0 := 0$. Define these two processes on each $[\tau_n, \tau_{n+1})$, using induction by $n$. 

\smallskip

{\it Induction base:} On the time interval $[\tau_0, \tau_1)$, these are reflected diffusions without jumps on $\BR_+$. We can construct them so that $Z^{(x)}(t) \le Z^{(y)}(t)$ for $t \in [\tau_0, \tau_1)$, because the corresponding reflected diffusion on $\BR_+$ without jumps is stochastically ordered. 

\smallskip

{\it Induction step:} If the processes are defined on $[\tau_{n-1}, \tau_n)$ so that $Z^{(x)}(t) \le Z^{(y)}(t)$ for $t \in [\tau_{n-1}, \tau_n)$ a.s. Then by continuity $x_n := Z^{(x)}(\tau_n-) \le y_n := Z^{(y)}(\tau_n-)$ a.s. Generate $Z^{(x)}(\tau_n) \sim r^{-1}\nu_{x_n}(\cdot)$ and $Z^{(y)}(\tau_n) \sim r^{-1}\nu_{y_n}(\cdot)$ so that $Z^{(x)}(\tau_n) \le Z^{(y)}(\tau_n)$ a.s. This is possible by $\nu_{x_n}(\cdot) \preceq \nu_{y_n}(\cdot)$. Because the corresponding reflected diffusion without jumps is stochastically ordered, we can generate $Z^{(x)}$ and $Z^{(y)}$ on $(\tau_n, \tau_{n+1})$ as reflected diffusions without jumps such that $Z^{(x)}(t) \le Z^{(y)}(t)$ for $t \in [\tau_n, \tau_{n+1})$. This completes the proof by induction. 
\end{proof}

Next comes the central result of this paper: an explicit rate of exponential convergence for a reflected jump-diffusion on $\BR_+$. 

\begin{thm}
Consider a reflected jump-diffusion $Z = (Z(t), t \ge 0)$ on $\BR_+$ with a stochastically ordered family of jump measures $(\nu_x)_{x \in \BR_+}$. Under Assumptions 1, 2, 3, suppose for some $\la > 0$, 
\begin{equation}
\label{eq:ol-k}
K_{\max}(\la) := \sup\limits_{x > 0}K(x, \la) < 0.
\end{equation}
Then the process $Z$ is $V_{\la}$-uniformly ergodic with the exponential rate of convergence $\vk = |K_{\max}(\la)|$. The (unique) stationary distribution $\pi$ satisfies $(\pi, V_{\la}) < \infty$. 
\label{thm:3}
\end{thm}

\begin{rmk} In certain cases, this exponential rate of convergence is exact: one cannot increase the value of $\vk$ for the given norm $\norm{\cdot}_{V_{\la}}$ for fixed $\la$; see Lemma~\ref{lemma:exact-rate} in Section 6.
\end{rmk}

\begin{proof} That this process has a unique stationary distribution $\pi$ with $(\pi, V_{\la_0}) < \infty$ follows from Theorem~\ref{thm:1}. In light of Lemma~\ref{lemma:stoch-RJD}, to complete the proof of Theorem~\ref{thm:3}, let us show that $V_{\la_0}$ is a modified Lyapunov function. Take an $\eta > 0$ and let 
$\tau(\eta) := \inf\{t \ge 0\mid Z(t) \le \eta\}$. Let us show that the following process is a local supermartingale:
\begin{equation}
\label{eq:eta}
V_{\la_0}(Z(t\wedge\tau(\eta))) + |K_{\max}(\la)|\int_0^{t\wedge\tau(\eta)}V_{\la_0}(Z(s))\md s,\ \ t \ge 0.
\end{equation}
Indeed, take a function $\ol{V}_{\la}$ from~\eqref{eq:ol-V} with the function $\phi$ from~\eqref{eq:phi} constructed so that $s_2 < \eta$. From~\eqref{eq:CL-V}, we have: $\CL\ol{V}_{\la}(x) \le K(x, \la)\ol{V}_{\la}(x)$ for $x \ge \eta$. But $K(x, \la) \le K_{\max}(\la) < 0$. Therefore,
\begin{equation}
\label{eq:est-on-k}
\CL\ol{V}_{\la}(x) \le K_{\max}(\la)\ol{V}_{\la}(x) = -|K_{\max}(\la)|\ol{V}_{\la}(x),\ x \ge \eta.
\end{equation}
The following process is a local martingale:
$$
\ol{V}_{\la}(Z(t)) - \int_0^t\CL\ol{V}_{\la}(Z(s))\md s,\ t \ge 0.
$$
By the optional stopping theorem, the following process is also a local martingale:
\begin{equation}
\label{eq:new}
\ol{V}_{\la}(Z(t\wedge\tau(\eta))) - \int_0^{t\wedge\tau(\eta)}\CL\ol{V}_{\la}(Z(s))\md s,\ \ t \ge 0.
\end{equation}
Observe that  $\ol{V}_{\la}(x) = V_{\la}(x)$ for $x \ge \eta$, but $Z(s) \ge \eta$ for $s < \tau(\eta)$. Combining this with~\eqref{eq:new} and~\eqref{eq:est-on-k}, we get that the process in~\eqref{eq:eta} is a local supermartingale. It suffices to let $\eta \downarrow 0$ and observe that $\tau(\eta) \uparrow \tau(0)$. Therefore, for $\eta = 0$ the process in~\eqref{eq:eta} is also a local supermartingale. Actually, it is a true supermartingale, because it is bounded from below (apply Fatou's lemma). Apply Theorem~\ref{thm:2}, observe that the function $V_{\la}$ is nondecreasing, and complete the proof. 
\end{proof}

The next corollary is analogous to Corollary~\ref{cor:m<0}. Its proof is also similar to that of Corollary~\ref{cor:m<0} and is omitted. 

\begin{cor} 
Consider a reflected jump-diffusion on $\BR_+$ from Section 2, with a stochastically ordered family $(\nu_x)_{x \in \BR_+}$ of jump measures. 
Under Assumptions 1, 2, 3, if
$$
\sup\limits_{x \ge 0}m(x) < 0,\ \ \sup\limits_{x \ge 0}\si^2(x) < \infty,
$$
then there exists a $\la_0 > 0$ such that $K_{\max}(\la_0) < 0$, in the notation of~\eqref{eq:ol-k}. Therefore, the reflected jump-diffusion is $V_{\la_0}$-uniformly ergodic with exponential rate of convergence $\vk = |K_{\max}(\la_0)|$. The (unique) stationary distribution $\pi$ satisfies $(\pi, V_{\la_0}) < \infty$.
\end{cor}

\begin{exm} Consider the case when $\nu_x \equiv 0$: there are no jumps,  this process is a reflected diffusion on the positive half-line. Assume 
$$
\sup\limits_{x > 0}g(x) = -\ol{g} < 0,\ \ \si(x) \equiv 1.
$$
Then in the notation of~\eqref{eq:ol-k}, we can calculate
$$
K_{\max}(\la) = \sup\limits_{x > 0}K(x, \la) = \sup\limits_{x > 0}\left[g(x)\la + \frac{\la^2}2\right] = -\ol{g}\la + \frac{\la^2}2.
$$
This function $K_{\max}(\la)$ assumes its minimum value 
$-\ol{g}^2/2$ at $\la_* = \ol{g}$. Therefore, this reflected diffusion is $V_{\ol{g}}$-uniformly ergodic with exponential rate of convergence $\vk = \ol{g}^2/2$. This includes the case of reflected Brownian motion on the half-line with negative  drift from \cite[Section 6]{LMT1996}. 
\end{exm} 

\begin{exm} Consider the case $g(x) \equiv -2$, $\si(x) \equiv 1$, and $\nu_x = \de_{x+1}$ for $x \ge 0$. In other words, this reflected jump-diffusion has constant negative drift $-2$, constant diffusion $1$, and it jumps with rate $1$; each jump is one unit to the right. Assumption 3 holds with any $\la_0$. The negative drift ``outweighs'' the jumps in the positive direction: $m(x) = -2 + 1 = -1$. We have:
$$
K(x, \la) \equiv K(\la) = -2\la + \frac{\la^2}2 + e^{\la} - 1.
$$
For every $\la > 0$ such that $K(\la) < 0$, this process is $V_{\la}$-uniformly ergodic with exponential rate of convergence $\vk = |K(\la)|$. It is easy to calculate that $K(\la) < 0$ for $\la \in (0, 0.849245)$. For example, the function $K(\la)$ attains minimum value $-0.230503$ at $\la_* = 0.442954$.   Therefore, this reflected jump-diffusion is $V_{\la_*}$-uniformly ergodic with exponential rate of convergence $\vk := 0.230503$. Choosing larger values of $\la$ such that $K(\la) < 0$ (say, $\la = 0.8$) results in lower rate of exponential convergence, but the norm $\norm{\cdot}_{V_{\la}}$ which measures convergence becomes stronger. 
\end{exm}

\begin{exm} Consider a reflected jump-diffusion with the same drift and diffusion coefficients as in Example 1, but with $\nu_x(\md y) = 1_{\{y > x\}}e^{x-y}\md y$. In other words, the jumps occur with rate $\nu_x(\BR_+) = 1$, but each jump is to the right with the magnitude distributed as $\Exp(1)$. Then Assumption 3 holds with $\la_0 = 1$, and $m(x) = -2 + 1 = -1$. We have:
$$ 
K(x, \la) \equiv K(\la) = -2\la + \frac{\la^2}2 + \frac{1}{1 - \la} - 1. 
$$
This attains minimum value $-0.135484$ at $\la_* = 0.245122$. Therefore, this reflected jump-diffusion is $V_{\la_*}$-uniformly ergodic with exponential rate of convergence $\vk = 0.135484$. 
\end{exm}

\section{The Case of Non-Stochastically Ordered Processes}

If the reflected jump-diffusion is not stochastically ordered, then we can still sometimes estimate the exponential rate of convergence. This is the case when this process is stochastically dominated by another reflected jump-diffusion, which, in turn, is stochastically ordered. 

\begin{defn} Take two Markov processes $X = (X(t), t \ge 0)$, $\ol{X} = (\ol{X}(t), t \ge 0)$ with transition kernels $(P^t)_{t \ge 0}, (\ol{P}^t)_{t \ge 0}$ on $\BR_+$. We say that $X$ is {\it stochastically dominated by} $\ol{X}$ if $P^t(x, \cdot) \preceq \ol{P}^t(x, \cdot)$ for all $t \ge 0$ and $x \in \BR_+$. In this case, we write $X \preceq \ol{X}$. 
\end{defn}

The following auxillary statement is proved similarly to Lemma~\ref{lemma:stoch-RJD}.

\begin{lemma} Take two reflected jump-diffusions $Z$ and $\ol{Z}$ on $\BR_+$ with common drift and diffusion coefficients $g$ and $\si^2$, which satisfy Assumption 1, and with families $(\nu_x)_{x \in \BR_+}$ and $(\ol{\nu}_x)_{x \in \BR_+}$ of jump measures satisfying Assumption 2. Assume that $\nu_x \preceq \ol{\nu}_x$ for every $x \in \BR_+$, and the family $(\ol{\nu}_x)_{x \in \BR_+}$ is stochastically ordered. Then $Z \preceq \ol{Z}$. 
\label{lemma:stoch-comp-RJD}
\end{lemma}

The next result is an improvement upon \cite[Theorem 3.4]{LMT1996}. We prove convergence in $\norm{\cdot}_{V}$-norm, that is, uniform ergodicity, as opposed to convergence in the total variation norm, which was done in \cite[Theorem 3.4]{LMT1996}. In particular, if $V = V_{\la}$, as is the case for reflected jump-diffusions, then we can estimate the convergence rate for moments, as in Remark~\ref{rmk:moments}. Such estimation is impossible when one has convergence only in the total variation norm.

\begin{thm} Take a (possibly non-stochastically ordered) Markov process $X = (X(t), t \ge 0)$ which is stochastically dominated by another stochastically ordered Markov process $\ol{X} = (\ol{X}(t), t \ge 0)$. Assume $\ol{X}$ has a modified nondecreasing Lyapunov function $V$ with Lyapunov constant $k$.

\medskip

(a) Then for every $x_1, x_2 \in \BR_+$, we have: 
\begin{equation}
\label{eq:difference-pt-new}
\norm{P^t(x_1, \cdot) - P^t(x_2, \cdot)}_V \le \left[V(x_1) + V(x_2)\right]e^{-kt},\ \ t \ge 0.
\end{equation}

\medskip

(b) For initial distributions $\mu_1$ and $\mu_2$ on $\BR_+$ with $(\mu_1, V) < \infty$ and $(\mu_2, V) < \infty$, we have:
\begin{equation}
\label{eq:difference-distributions-pt-new}
\norm{\mu_1P^t - \mu_2P^t}_V \le \left[(\mu_1, V) + (\mu_2, V)\right]e^{-kt},\ \ t \ge 0.
\end{equation}

\medskip

(c) If the process $X$ has a stationary distribution $\pi$ which satisfies $(\pi, V) < \infty$, then this stationary distribution is unique, and the process $X$ is $V$-uniformly ergodic with exponential rate of convergence $\vk = k$. More precisely, we have the following estimate:
\begin{equation}
\label{eq:difference-distributions-pt-new}
\norm{P^t(x, \cdot) - \pi}_V \le \left[(\pi, V) + V(x)\right]e^{-kt},\ \ t \ge 0;
\end{equation}
\label{thm:non-stoch-ordered}
\end{thm}

\begin{proof} Let us show (a). We combine the proofs of \cite[Theorem 2.2]{LMT1996}, \cite[Theorem 4.1]{LT1996}, and modify them a bit. Without loss of generality, assume $x_2 \le x_1$. Consider two copies $X_1, X_2$ of the process $X$, and two copies $\ol{X}_1, \ol{X}_2$ of the process $\ol{X}$, starting from
$$
X_1(0) = \ol{X}_1(0) = x_1,\ \ X_2(0) = \ol{X}_2(0) = x_2.
$$
Take a measurable function $g : \BR_+ \to \BR$ such that $|g| \le V$, and estimate from above the difference
\begin{equation}
\label{eq:difference}
\left|\ME g(X_1(t)) - \ME g(X_2(t))\right|.
\end{equation}
%This is estimates as follows:
%$$
%\left|\ME g(X_1(t)) - \ME g(X_2(t))\right| \le \left|\ME g(X_1(t)) - \ME g(X_0(t))\right| + \left|\ME g(X_2(t)) - \ME g(X_0(t))\right|.
%$$
Because of stochastic ordering, using $x_2 \le x_1$, we can couple these processes so that 
\begin{equation}
\label{eq:stoch-coupling}
X_2(t) \le \ol{X}_2(t) \le \ol{X}_1(t),\ \ X_2(t) \le X_1(t) \le \ol{X}_1(t).
\end{equation}
Define the stopping time $\ol{\tau}(0) := \inf\{t \ge 0\mid \ol{X}_1(t) = 0\}$. By~\eqref{eq:stoch-coupling}, $X_1(\ol{\tau}(0)) = X_2(\ol{\tau}(0)) = 0$, so $\ol{\tau}(0)$ is a (random) coupling time for $X_1$ and $X_2$: the laws of $(X_1(t), t \ge \ol{\tau}(0))$ and $(X_2(t), t \ge \ol{\tau}(0))$ are the same. Therefore, 
$$
\ME g(X_1(t))1_{\{t > \ol{\tau}(0)\}} = \ME g(X_2(t))1_{\{t > \ol{\tau}(0)\}},
$$
and the quantity from~\eqref{eq:difference} can be estimated from above by 
\begin{equation}
\label{eq:estimate-from-above}
\left|\ME g(X_1(t))1_{\{t \le \ol{\tau}(0)\}} - \ME g(X_2(t))1_{\{t \le \ol{\tau}(0)\}}\right| \le 
\ME|g(X_1(t))|1_{\{t \le \ol{\tau}(0)\}} + \ME|g(X_2(t))|1_{\{t \le \ol{\tau}(0)\}}.
\end{equation}
Let us estimate the first term in the right-hand side of~\eqref{eq:estimate-from-above}. Because $|g| \le V$, we have:
\begin{equation}
\label{eq:|g|}
\ME|g(X_1(t))|1_{\{t \le \ol{\tau}(0)\}} \le \ME V(X_1(t))1_{\{t \le \ol{\tau}(0)\}}.
\end{equation}
Next, because the function $V$ is nondecreasing, 
\begin{equation}
\label{eq:chain-of-ineq}
e^{kt}\ME V(X_1(t))1_{\{t \le \ol{\tau}(0)\}} \le e^{kt}\ME V(\ol{X}_1(t))1_{\{t \le \ol{\tau}(0)\}} \le \ME\left[e^{k(t\wedge\ol{\tau}(0))}V\left(\ol{X}_1(t\wedge\ol{\tau}(0))\right)\right].
\end{equation}
Let us show that the following process is a supermartingale:
\begin{equation}
\label{eq:tilde-M}
\tilde{M}(t) = e^{k(t\wedge\tau)}V\left(\ol{X}_1(t\wedge\tau)\right),\ \ t \ge 0.
\end{equation}
Indeed, from~\eqref{eq:M-mart}, we already know that the following process is a supermartingale:
$$
M(t) = V(\ol{X}_1(t\wedge\ol{\tau}(0))) + k\int_0^{t\wedge\ol{\tau}(0)}V(\ol{X}_1(s))\md s,\ \ t \ge 0.
$$
Applying Ito's formula to $\tilde{M}(t)$ in~\eqref{eq:tilde-M} for $t \le \ol{\tau}(0)$, we have:
$$
\md\tilde{M}(t) = ke^{kt}V\left(\ol{X}_1(t)\right)\md t + e^{kt}\md V\left(\ol{X}_1(t)\right) = e^{kt}\md M(t).
$$
This is also true for $t \ge \ol{\tau}(0)$, because both $M$ and $\tilde{M}$ are constant on $[\ol{\tau}(0), \infty)$. Since $M$ is a supermartingale, it can be represented as $M(t) = M_1(t) + M_2(t)$, where $M_1$ is a local martingale, and $M_2$ is a nonincreasing process. Therefore, 
$$
\tilde{M}(t) = \int_0^te^{ks}\md M_1(s) + \int_0^te^{ks}\md M_2(s) =: \tilde{M}_1(t) + \tilde{M}_2(t)
$$
is also a sum of a local martingale and a nonincreasing process. Thus, it is a local supermartingale. Actually, it is a true supermartingale, because it is nonnegative (use Fatou's lemma). Therefore,
\begin{equation}
\label{eq:expectations-tildeM}
\ME\tilde{M}(t) \le \ME\tilde{M}(0) = V(x_1).
\end{equation}
Comparing~\eqref{eq:|g|},~\eqref{eq:chain-of-ineq},~\eqref{eq:tilde-M},~\eqref{eq:expectations-tildeM}, we have: $\ME|g(X_1(t))|1_{\{t \le \ol{\tau}(0)\}} \le V(x_1)$. Similarly estimate the second term in the right-hand side of~\eqref{eq:estimate-from-above}, and combine this with~\eqref{eq:estimate-from-above}:
\begin{equation}
\label{eq:complete-estimate}
\left|\ME g(X_1(t)) - \ME g(X_2(t))\right| \le \left[V(x_1) + V(x_2)\right]e^{-kt},\ \ t \ge 0.
\end{equation}
Taking the supremum in~\eqref{eq:complete-estimate} over $|g| \le V$, we complete the proof of~\eqref{eq:difference-pt-new}.

\medskip

(b) Integrate over $(x_1, x_2) \sim \mu_1\times\mu_2$ in~\eqref{eq:complete-estimate} and take the supremum over $|g| \le V$.

\medskip

(c) Apply (b) to $\mu_1 = \pi$ and $\mu_2 = \de_{x}$. Since $V(x) \ge 1$, we can take $D = 1 + (\pi, V)$ in~\eqref{eq:uniform-ergodicity}.
\end{proof}

Now, we apply Theorem~\ref{thm:non-stoch-ordered} to reflected jump-diffusions. 

\begin{cor} Take drift and diffusion coefficients $g$, $\si^2$, satisfying Assumption 1. Take two families $(\nu_x)_{x \in \BR_+}$ and $(\ol{\nu}_x)_{x \in \BR_+}$ of jump measures which satisfy Assumptions 2 and 3, such that $\nu_x \preceq \ol{\nu}_x$ for every $x \in \BR_+$, and the family $(\ol{\nu}_x)_{x \in \BR_+}$ is stochastically ordered. Consider a reflected jump-diffusion on $\BR_+$ with drift and diffusion coefficients $g, \si^2$, and the family $(\nu_x)_{x \ge 0}$ of jump measures. Let 
$$
\ol{K}(x, \la) = g(x)\la + \si^2(x)\frac{\la^2}2 + \int_0^{\infty}\left[e^{\la(y-x)} - 1\right]\ol{\nu}_x(\md y).
$$
Assume there exists a $\la_* > 0$ such that 
\begin{equation}
\label{eq:sup-aux}
\sup\limits_{x > 0}\ol{K}(x, \la_*) =: \ol{K}_{\max}(\la_*) < 0.
\end{equation}
Then $Z$ is $V_{\la_*}$-uniformly ergodic with exponential rate of convergence $\vk = |\ol{K}_{\max}(\la_*)|$. 
\label{cor:non-stoch-ordered-RJD}
\end{cor}

\begin{proof} For each $x \in \BR_+$, the function $y \mapsto e^{\la_*(y-x)} - 1$ is nondecreasing, and $\nu_x \preceq \ol{\nu}_x$. Therefore,
$$
\int_0^{\infty}\left[e^{\la_*(y-x)} - 1\right]\nu_x(\md y) \le \int_0^{\infty}\left[e^{\la_*(y-x)} - 1\right]\ol{\nu}_x(\md y).
$$
This, in turn, implies that for $x \in \BR_+$, 
\begin{equation}
\label{eq:k-ol-k}
K(x, \la_*) = g(x)\la_* + \si^2(x)\frac{\la_*^2}2 + \int_0^{\infty}\left[e^{\la_*(y-x)} - 1\right]\nu_x(\md y) \le \ol{K}(x, \la_*).
\end{equation}
Comparing~\eqref{eq:sup-aux} with~\eqref{eq:k-ol-k}, we get: 
$$
\sup\limits_{x > 0}K(x, \la_*) \le \sup\limits_{x > 0}\ol{K}(x, \la_*) < 0.
$$
By Theorem~\ref{thm:3}, the process $Z$ is $V_{\la_*}$-uniformly ergodic, and its stationary distribution $\pi$ satisfies $(\pi, V_{\la_*}) < \infty$. Consider another reflected jump-diffusion $\ol{Z}$ with the same drift and diffusion coefficients $g, \si^2$, and the family $(\ol{\nu}_x)_{x \in \BR_+}$ of jump measures. By Lemma~\ref{lemma:stoch-comp-RJD}, $Z \preceq \ol{Z}$. Similarly to Theorem~\ref{thm:2}, we can show that $V_{\la_*}$ is a modified Lyapunov function for $\ol{Z}$. Applying Theorem~\ref{thm:non-stoch-ordered} and using $V_{\la_*}$ as a modified Lyapunov function, we complete the proof. 
\end{proof}

\begin{exm} Take a continuous function $\psi : \BR_+ \to \BR_+$ such that $\psi(x) \le x + 1$. Consider a reflected jump-diffusion on $\BR_+$ with $g(x) = -2$, $\si^2(x) = 1$, and the family of jump measures $(\nu_x)_{x \ge 0}$, with $\nu_x := \de_{\psi(x)}$. This family is not necessarily stochastically ordered (because $\psi$ is not necessarily nondecreasing). However, $\nu_x \preceq \de_{x+1}$, and we can apply Corollary~\ref{cor:non-stoch-ordered-RJD}. We have:
$$
\ol{K}(x, \la) = -2\la + \frac{\la^2}2 + e^{\la} - 1   
$$
has minimum value $-0.230503$ at $\la_* = 0.442954$.   Therefore, this reflected jump-diffusion is $V_{\la_*}$-uniformly ergodic with exponential rate of convergence $\vk := 0.230503$.
\end{exm}

\section{The Best Exponential Rate of Convergence} Consider a reflected jump-diffusion $Z = (Z(t), t \ge 0)$ with constant drift and diffusion coefficients: $g(x) \equiv g$, $\si^2(x) \equiv \si^2$, and with family of jump measures $(\nu_x)_{x \ge 0}$ defined by $\nu_x(E) = \mu((E-x)\cap\BR_+)$ for $E\subseteq\BR_+$, where $\mu$ is a finite Borel measure supported on $\BR_+$. In other words, every $\nu_x$ is the push-forward of the measure $\mu$ with respect to the mapping $y \mapsto x + y$. This process is a reflected Brownian motion on $\BR_+$ with jumps, which are directed only to the right, with the magnitude and the intensity independent of $x$. 

Recall that the intensity of jumps originating from a point $x \in \BR_+$ is equal to $r(x) := \nu_x(\BR_+)$, and its magnitude is distributed as $|y-x|$, where $y \sim r^{-1}(x)\nu_x(\cdot)$. In this case,  the intensity of jumps is equal to $r = \mu(\BR_+)$, and the magnitude is distributed according to the normalized measure $r^{-1}\mu(\cdot)$. This was the case in Examples 2 and 3 from Section 4. 

Then the family of jump measures $(\nu_x)_{x \ge 0}$ is stochastically ordered. Next, 
$$
K(x, \la) = K(\la) = g\la + \frac{\si^2}2\la^2 + \int_{\BR_+}\left[e^{\la z} - 1\right]\mu(\md z).
$$
From Theorem~\ref{thm:3}, we know that if 
\begin{equation}
\label{eq:drift<0}
g + \int_{\BR_+}z\mu(\md z) < 0,
\end{equation}
then there exists a $\la > 0$ such that $K(\la) < 0$, and the reflected jump-diffusion is $V_{\la}$-uniformly ergodic with $\vk = |K(\la)|$. Actually, this rate of convergence is exact: one cannot improve this result. More precisely, for this $\la$, one cannot find a $\vk > |K(\la)|$ such that the reflected jump-diffusion is $V_{\la}$-uniformly ergodic with exponential rate of convergence $\vk$. This is a counterpart of \cite[Theorem 2.3]{LMT1996}, which finds the exact exponential rate of convergence in the total variation metric. Unfortunately, we cannot apply their results, because they require $\pi(\{0\}) > 0$ for a stationary distribution $\pi$, which is not true in our case. As mentioned in Example 2, we can make a trade-off between the rate of convergence and the strength of the norm $\norm{\cdot}_{V_{\la}}$. 

\begin{lemma} Under the condition~\eqref{eq:drift<0}, for every $\la \in (0, \la_0)$ such that $K(\la) < 0$, and every $x_1, x_2 \in \BR_+$, we have:
$$
\ME_{x_1}V_{\la}(Z(t)) - \ME_{x_2}V_{\la}(Z(t)) = (V_{\la}(x_1) - V_{\la}(x_2))e^{-|K(\la)|t},\ \ t \ge 0.
$$
\label{lemma:exact-rate}
\end{lemma}

\begin{proof} Let $\vk = |K(\la)|$. We must go over the proofs of Theorems~\ref{thm:1},~\ref{thm:2}, and~\ref{thm:non-stoch-ordered}. Using the notation from these theorems with an adjustment: $\ol{\eta}(0) = \eta(0)$, we get:
\begin{align*}
\ME_{x_1}V_{\la}(Z(t)) &- \ME_{x_2}V_{\la}(Z(t)) = \ME V_{\la}(X_1(t)) - \ME V_{\la}(X_2(t)) \\ & = \ME V_{\la}(X_1(t))1_{\{\tau(0) > t\}} - \ME V_{\la}(X_2(t))1_{\{\tau(0) > t\}} \\ & =  \ME V_{\la}(X_1(t\wedge\tau(0)))1_{\{\tau(0) > t\}} - \ME V_{\la}(X_2(t\wedge\tau(0)))1_{\{\tau(0) > t\}}.
\end{align*}
Multiplying by $e^{\vk t}$, we get:
\begin{align*}
e^{\vk t}&\left[\ME_{x_1}V_{\la}(Z(t)) - \ME_{x_2}V_{\la}(Z(t))\right] \\ & = \ME\left[e^{\vk (t\wedge\tau(0))}V_{\la}(X_1(t\wedge\tau(0)))1_{\{\tau(0) > t\}}\right] - \ME\left[e^{\vk (t\wedge\tau(0))}V_{\la}(X_2(t\wedge\tau(0)))1_{\{\tau(0) > t\}}\right] \\ & = 
\ME\left[e^{\vk(t\wedge\tau(0))}V_{\la}(X_1(t\wedge\tau(0)))\right] - \ME\left[e^{\vk(t\wedge\tau(0))}V_{\la}(X_2(t\wedge\tau(0)))\right],
\end{align*}
because on the event $\{t \ge \tau(0)\}$ we have: $X_1(t\wedge\tau(0)) = X_2(t\wedge\tau(0)) = 0$. Next, if we show that 
\begin{equation}
\label{eq:new-M}
\ol{M}(t) = e^{\vk(t\wedge\tau(0))}V_{\la}(Z(t\wedge\tau(0))),\ \ t \ge 0,
\end{equation}
is a martingale for every initial condition $Z(0) = x \in \BR_+$, then the rest of the proof is trivial: just use $\ME_{x}\ol{M}(t) = \ol{M}(0) = V_{\la}(x)$ for $x = x_1$ and $x = x_2$. Let us show that~\eqref{eq:new-M} is a martingale. We follow the proof of Theorems~\ref{thm:1},~\ref{thm:2}. If $x > s_2$, where $s_2$ is taken from~\eqref{eq:phi}, then we have equality in~\eqref{eq:V-vs-ol-V}, in~\eqref{eq:CN-V}, and in~\eqref{eq:CL-V}. Indeed, take an $x > s_2$, and let $y \sim r^{-1}(x)\nu_x(\cdot)$. Then $y - x \sim \mu$, therefore $y - x \ge 0$, and $y > s_2$, $\phi(y) = y$. Therefore, as in Theorem~\ref{thm:2}, the process
$$
V_{\la}(Z(t\wedge\tau(\eta))) + \vk\int_0^{t\wedge\tau(\eta)}V_{\la}(Z(s))\md s,\ \ t \ge 0,
$$
is a local martingale for every $\eta > 0$, and hence for $\eta = 0$, because $\tau(\eta) \uparrow \tau(0)$. Similarly to the proof of Theorem~\ref{thm:non-stoch-ordered}, we can show that~\eqref{eq:new-M} is a local martingale. Actually, it is a true martingale. Indeed, take an $\eps := \la_0/\la - 1 > 0$. Then for all $x \in \BR_+$ and $t > 0$, 
\begin{equation}
\label{eq:finite-exp}
\ME_x\sup\limits_{0 \le s \le t}\left[V_{\la}(Z(s))\right]^{1 + \eps} < \infty.
\end{equation}
Indeed, we can represent $Z(s) = B(s) + \sum_{i=1}^{\mathcal J(s)}\xi_i$, where $B = (B(s), s \ge 0)$ is a reflected Brownian motion on $\BR_+$ with drift and diffusion coefficients $g$ and $\si^2$, starting from $B(0) = x$, random variables $\xi_i \sim r^{-1}\mu(\cdot)$ are i.i.d., $\mathcal J = (\mathcal J(s), s \ge 0)$ is a Poisson process on $\BR_+$ with constant intensity $r$, and $B, \mathcal J, \xi_i$ are independent. Then
\begin{equation}
\label{eq:auxillary-decomp}
\sup\limits_{0 \le s \le t}\left[V_{\la}(Z(s))\right]^{1 + \eps} = \exp\Bigl(\la_0\max\limits_{0 \le s \le t}B(s)\Bigr)\exp\Bigl(\la_0\sum_{i=1}^{\mathcal J(t)}\xi_i\Bigr). 
\end{equation}
The moment generating function $G_{\xi}(u) := \ME e^{u\xi}$ of $\xi_i$ is finite for $u = \la_0$. Therefore, the moment generating function of the random sum of random variables is equal to
$$
G(u) := \ME \exp\Bigl(u\sum_{i=1}^{\mathcal J(t)}\xi_i\Bigr) = \exp\left(r(G_{\xi}(u) - 1)\right).
$$
This quantity is also finite for $u = \la_0$. Finally, $\ME\exp\left(\la_0\max_{0 \le s \le t}B(s)\right) < \infty$. Apply~\eqref{eq:auxillary-decomp} and complete the proof of~\eqref{eq:finite-exp} together with the martingale property of~\eqref{eq:new-M} and Lemma~\ref{lemma:exact-rate}. 
\end{proof}

\section{Systems of Two Competing Levy Particles}

\subsection{Motivation and historical review} Finite systems of rank-based competing Brownian particles on the real line, with drift and diffusion coefficients depending on the current rank of the particle relative to other particles, were introduced in \cite{BFK2005} as a model for mathematical finance. Similar systems with Levy processes instead of Brownian motions were introduced in \cite{Shkolnikov2011}. One important question is {\it stability}: do the particles move together as $t \to \infty$? We can restate this question in a different way: Consider the gaps between consecutive ranked particles; do they converge to some stationary distribution as $t \to \infty$, and if yes, how fast? For competing Brownian particles, this question was resolved in \cite[Proposition 2]{Ichiba11}, \cite[Proposition 2.2]{MyOwn6}, \cite[Proposition 4.1]{MyOwn10}: necessary and sufficient conditions were found for stability. If the system is indeed stable, then the gap process has a unique stationary distribution $\pi$, and it converges to $\pi$ exponentially fast. 

However, it is a difficult question to explicitly estimate the rate of this exponential convergence; it was done in a particular case of unit diffusion coefficients in \cite{IPS2013}. For competing Levy particles, partial results for convergence were obtained in \cite{Shkolnikov2011}. However, an explicit estimate of the rate of exponential convergence remains unknown. In this section, we consider systems of two competing Levy particles. We improve the convergence conditions of \cite{Shkolnikov2011}. In some cases, we are able to find an explicit rate of convergence.

\subsection{Definition and construction} Take a drift vector and a positive definite symmetric matrix 
\begin{equation}
\label{eq:original-drift-diffusion}
(g_+, g_-) \in \BR^2,\ \ 
A = 
\begin{bmatrix} 
a_{++} & a_{+-}\\ a_{+-} & a_{--}
\end{bmatrix}
\end{equation}
Take a finite Borel measure $\Lambda$ on $\BR^2$. Consider a Levy process 
$L(t) = (L_+(t), L_-(t))'$, $t \ge 0$, on the space $\BR^2$, with drift vector $(g_+, g_-)'$, covariance matrix $A$, and jump measure $\Lambda$. Take two real-valued r.c.l.l. (right-continuous with left limits) processes $X_1(t), X_2(t), t \ge 0$, which satisfy the following system of equations:
\begin{equation}
\label{eq:main-CLP}
\begin{cases}
\md X_1(t) = 1\left(X_1(t) > X_2(t)\right)\md L_+(t) + 1\left(X_1(t) \le X_2(t)\right)\md L_-(t);\\
\md X_2(t) = 1\left(X_1(t) \le X_2(t)\right)\md L_+(t) + 1\left(X_1(t) > X_2(t)\right)\md L_-(t).
\end{cases}
\end{equation}
At each time $t \ge 0$, we rank the particles $X_1(t)$ and $X_2(t)$:
$$
Y_+(t) = X_1(t)\vee X_2(t),\ \ Y_-(t) = X_1(t)\wedge X_2(t),\ \ t \ge 0.
$$
In case of a tie: $X_1(t) = X_2(t)$, we assign to $X_2(t)$ the higher rank. (We say that ties are resolved in lexicographic order.) At this time $t$, the lower-ranked particle behaves as the process $L_-$, and the higher-ranked particle behaves as the process $L_+$. 

%If $a_{+-} \ne 0$, there is dependence between diffusion parts of these two Levy components, which are driving the lower- and the higher-ranked particles. 

\begin{rmk} 
\label{rmk:indep-jumps}
Assume there exist finite Borel measures $\nu_-, \nu_+$ on $\BR$ such that 
$$
\Lambda(\md x_+, \md x_-) = \delta_{0}(\md x_+)\times\nu_-(\md x_-) + \nu_+(\md x_+)\times\delta_{0}(\md x_-), 
$$
then the jumps of the two particles occur independently. The jumps of the lower-ranked particle $Y_-(t)$ and the higher-ranked particle $Y_+(t)$ are governed by measures $\nu_-$ and $\nu_+$ respectively. 
\end{rmk}

\begin{lemma} There exists in the weak sense a unique in law solution to the system~\eqref{eq:main-CLP}. Moreover, consider the {\it gap process} $Z(t) = Y_+(t) - Y_-(t),\, t \ge 0$. This is a reflected jump-diffusion on $\BR_+$ with drift and diffusion coefficients 
\begin{equation}
\label{eq:drift-diffusion}
g = g_+ - g_-\ \ \mbox{and}\ \ \si^2 = a_{++} + a_{--} - 2a_{+-},
\end{equation}
and with the family $(\nu_z)_{z \in \BR_+}$ of jump measures, where for every $z \in \BR_+$, the measure $\nu_z$ is defined as the push-forward of the measure $\Lambda$ under the mapping $F_z : (x_+, x_-) \mapsto |x_+ - x_- + z|$. 
\label{lemma:gap-RJD}
\end{lemma}

\begin{rmk} Here, we construct a slightly more general version of a system of competing Levy particles than in \cite{Shkolnikov2011}. Indeed, in \cite{Shkolnikov2011} they assume that the jumps are independent, as in Remark~\ref{rmk:indep-jumps}; moreover, the diffusion parts are also uncorrelated, $a_{+-} = 0$, and $\nu_+ = \nu_-$. 
\end{rmk}

\begin{proof} Instead of writing all the formal details, which can be easily adapted from \cite{BFK2005, Shkolnikov2011}, let us informally explain  where the parameters from~\eqref{eq:drift-diffusion} come from and why the measure $\nu_z$ for each $z \in \BR_+$ is as described. Between jumps, the gap process moves as $B_+ - B_-$, where $(B_-, B_+)$ is a Brownian motion with drift vector and covariance matrix as in~\eqref{eq:original-drift-diffusion}. It is easy to calculate that $B_+ - B_-$ is a one-dimensional Brownian motion with drift and diffusion coefficients from~\eqref{eq:drift-diffusion}. 

Next, let us show the statement about the family of jump measures. Assume that $\tau$ is the moment of a jump, and immediately before the jump, the gap between two particles was equal to $z$: $Z(\tau-) = z$. Then $y_+ = Y_+(\tau-)$ and $y_- = Y_-(\tau-)$ satisfy $y_+ - y_- = z$. Assume without loss of generality that $X_1(\tau-) = y_-$ and $X_2(\tau-) = y_+$. (This choice is voluntary when there is no tie: $y_- \ne y_+$, but required if there is a tie: $y_- = y_+$, because ties are resolved in lexicographic order.) The displacement $(x_-, x_+)$ during the jump is distributed according to the normalized measure $\Lambda$, or, more precisely, $\left[\Lambda(\BR^2)\right]^{-1}\Lambda(\md x_+, \md x_-)$. After the jump, the positions of the particles are:
$$
X_1(\tau) = x_- + y_-,\ \ X_2(\tau) = x_+ + y_+.
$$
The new value of the gap process is given by
$$
Z(\tau) = |X_2(\tau) - X_1(\tau)| = |x_+ + y_+ - x_- - y_-| = |z + x_+ - x_-| = F_z(x_-, x_+). 
$$
Therefore, the destination of the jump of $Z$ from the position $z$ is distributed according to the probability measure $\tilde{\nu}_z$, which is the push-forward of the normalized measure 
$$
\left[\Lambda(\BR^2)\right]^{-1}\Lambda(\md x_+, \md x_-)
$$
with respect to the mapping $F_z$. The intensity of the jumps of $Z$ is constant and is always equal to the intensity of the jumps of the two-dimensional process $L$, that is, to $\Lambda(\BR^2)$. Therefore, the jump measure for $Z(t) = z$ is equal to the product of the intensity of jumps, which is $\Lambda(\BR^2)$, and the probability measure $\tilde{\nu}_z$. The rest of the proof is trivial. 
\end{proof}

\subsection{Uniform ergodicity of the gap process} Lemma~\ref{lemma:gap-RJD} allows us to apply previous results of this paper to this gap process. First, we apply Corollary~\ref{cor:m<0}. Recall the definition of the function $V_{\la}$ from~\eqref{eq:V}.

\begin{thm} Assume that
\begin{equation}
\label{eq:exp-moment}
\iint_{\BR^2}e^{\la_0(|x_+| + |x_-|)}\Lambda(\md x_+, \md x_-) < \infty\ \ \mbox{for some}\ \ \la_0 > 0;
\end{equation}
\begin{equation}
\label{eq:essential-drift<0}
g_+ - g_- + \iint_{\BR^2}\left[x_+ - x_-\right]\Lambda(\md x_+, \md x_-) < 0.
\end{equation}
Then for some $\la > 0$, the gap process is $V_{\la}$-uniformly ergodic, and the stationary distribution $\pi$ satisfies $(\pi, V_{\la}) < \infty$. 
\label{thm:CLP}
\end{thm}

We can rewrite the condition~\eqref{eq:essential-drift<0} as $m_+ < m_-$, where the magnitudes 
$$
m_+ = g_+ + \iint_{\BR^2}x_+\Lambda(\md x_+, \md x_-)\ \ \mbox{and}\ \ m_- = g_- + \iint_{\BR^2}x_-\Lambda(\md x_+, \md x_-)
$$
can be viewed as {\it effective drifts} of the upper- and the lower-ranked particles: the sum of the true drift coefficient and the mean value of the displacement during the jump, multiplied by the intensity of jumps. This is analogous to the stability condition for a system of two competing Brownian particles from \cite{Ichiba11}: the drift (in this case, the true drift) of the bottom particle must be strictly greater than the drift of the top particle. 

\begin{proof} Let us check conditions of Corollary~\ref{cor:m<0}. Assumption 1 and the boundedness of $\si^2$ are trivial. Assumption 2 follows from the fact that the function $F_z(x_+, x_-)$ is continuous in $z$. Assumption 3 follows from the condition~\eqref{eq:exp-moment}. Indeed, for $(x_+, x_-) \in \BR^2$ and $z \in \BR_+$, we have: $||x_+ - x_- + z| - z| \le |x_-| + |x_+|$. Therefore, we get: for every $z \in \BR_+$, 
$$
\int_{\BR_+}e^{\la_0|w-z|}\nu_z(\md w) = \iint_{\BR^2}e^{\la_0||x_+ - x_- + z| - z|}\Lambda(\md x_+, \md x_-) \le \iint_{\BR^2}e^{\la_0(|x_-| + |x_+|)}\Lambda(\md x_+, \md x_-) < \infty.
$$
Finally, let us check the condition~\eqref{eq:m<0}. Indeed, 
\begin{align*}
\int_{\BR_+}[w - z]\nu_z(\md w) &= \iint_{\BR^2}\left[|z + x_+ - x_-| - z\right]\Lambda(\md x_+, \md x_-) \\ & = \iint\limits_{\{x_- - x_+ > z\}}\left[x_- - x_+ - 2z\right]\Lambda(\md x_+, \md x_-) + \iint\limits_{\{x_- - x_+ \le z\}}\left[x_+ - x_-\right]\Lambda(\md x_+, \md x_-) \\ & = \iint\limits_{\{x_- - x_+ > z\}}\left[2x_- - 2x_+ - 2z\right]\Lambda(\md x_+, \md x_-) + \iint_{\BR^2}\left[x_+ - x_-\right]\Lambda(\md x_+, \md x_-).
\end{align*}
However, 
$$
\iint\limits_{\{x_- - x_+ > z\}}\left[2x_- - 2x_+ - 2z\right]\Lambda(\md x_+, \md x_-) \le 2\iint\limits_{\{x_- - x_+ > z\}}\left[x_- - x_+\right]\Lambda(\md x_+, \md x_-).
$$
There exists a constant $C_0 > 0$ such that $s \le C_0e^{\la_0s}$ for $s \ge 1$. Therefore, for $z \ge 1$, 
\begin{align*}
\iint\limits_{\{x_- - x_+ > z\}}&\left[x_- - x_+\right]\Lambda(\md x_+, \md x_-) \le C_0\iint\limits_{\{x_- - x+ > z\}}e^{\la_0(x_- - x_+)}\Lambda(\md x_+, \md x_-) \\ & \le C_0\iint\limits_{\{x_- - x+ > z\}}e^{\la_0(|x_-| + |x_+|)}\Lambda(\md x_+, \md x_-) \to 0
\end{align*}
as $z \to \infty$, because of~\eqref{eq:exp-moment}. Combining this with previous estimates, we get:
\begin{equation}
\label{eq:777}
\varlimsup\limits_{z \to \infty}\int_{\BR_+}[w - z]\nu_z(\md w) \le \iint_{\BR^2}\left[x_+ - x_-\right]\Lambda(\md x_+, \md x_-).
\end{equation}
Combining~\eqref{eq:777} with~\eqref{eq:essential-drift<0}, we complete the proof of~\eqref{eq:m<0}. 
\end{proof}

\begin{rmk} For the case of independent jumps, as in Remark~\ref{rmk:indep-jumps}, 
condition~\eqref{eq:exp-moment} can be written as
$$
\int_{-\infty}^{\infty}e^{\la_0|x_-|}\nu_-(\md x_-) < \infty\ \ \mbox{and}\ \ \int_{-\infty}^{\infty}e^{\la_0|x_+|}\nu_+(\md x_+) < \infty\ \ \mbox{for some}\ \ \la_0 > 0.
$$
and condition~\eqref{eq:essential-drift<0} can be written as
$$
g_+ + \int_{-\infty}^{\infty}x_+\nu_+(\md x_+) < g_- + \int_{-\infty}^{\infty}x_-\nu_-(\md x_-).
$$
\end{rmk}

\subsection{Explicit rate of exponential convergence} In some cases, we are able to find an explicit rate $\vk$ of exponential convergence for the gap process. This is true when the gap process is either stochastically ordered or dominated by some stochastically ordered reflected jump-diffusion. First, consider the case when the gap process is itself stochastically ordered. In addition to the assumptions of Theorem~\ref{thm:CLP}, let us impose the following assumption.

\begin{asmp}
The measure $\Lambda$ is supported on the subset $\{(x_+, x_-) \in \BR^2\mid x_+ \ge x_-\}$.
\end{asmp}

Then the family of jump measures $(\nu_z)_{z \in \BR_+}$ is stochastically ordered, because for $x_+ \ge x_-$, we have: $|z + x_+ - x_-| = z + x_+ - x_-$, and this quantity is increasing with respect to $z$. Therefore, we can apply Theorem~\ref{thm:3}. We have: 
$F_z(x_+, x_-) -z = |z + x_+ - x_-| - z = x_+ - x_-$. 
Applying the push-forward to the measure $\La$, for $g$ and $\si^2$ from~\eqref{eq:drift-diffusion}, we get:
\begin{align*}
K(z, \la) &= g\la + \frac{\si^2}2\la^2 + \int_{\BR_+}\left[e^{\la(w-z)} - 1\right]\nu_z(\md w) \\ & = g\la + \frac{\si^2}2\la^2 + \iint_{\BR^2}\left[e^{\la(x_+ - x_-)} - 1\right]\Lambda(\md x_+, \md x_-).
\end{align*}
This quantity does not actually depend on $z$, so we can denote it by $K(\la)$. If $k(\la) < 0$ for some $\la > 0$, then the gap process is $V_{\la}$-uniformly ergodic with exponent of ergodicity $\vk = |K(\la)|$. In particular, if we wish to maximize the rate $\vk$ of convergence, we need to minimize $K(\la)$. Actually, under Assumption 4, we are in the setting of Section 6, and $\vk = |K(\la)|$ gives the exact rate of exponential convergence in the $V_{\la}$-norm. 

\begin{rmk} For the case of independent jumps from Remark~\ref{rmk:indep-jumps}, Assumption 4 is equivalent to the condition that the measure $\nu_+$ is supported on $\BR_+$, the measure $\nu_-$ is supported on $\BR_-$, and 
$$
K(\la) = g\la + \frac{\si^2}2\la^2 + \int_{-\infty}^{\infty}\left[e^{\la x_+} - 1\right]\nu_+(\md x_+) + \int_{-\infty}^{\infty}\left[e^{-\la x_-} - 1\right]\nu_-(\md x_-).
$$
\end{rmk}

\begin{exm} Consider a system of two competing Levy particles with parameters 
$$
g_+ = 0,\ g_- = 3,\ a_{++} = a_{--} = 1,\ a_{+-} = 0,
$$
and measures $\nu_+$ on $\BR_+$ and $\nu_-$ on $\BR_-$ with densities
$$
\nu_+(\md x_+) = 1_{\{x_+ \ge 0\}}e^{-x_+}\md x_+,\ \ \nu_-(\md x_-) = 1_{\{x_- \le 0\}}e^{x_-}\md x_-.
$$
In other words, the upper particle can jump upwards, and the lower particle can jump downwards. For each particle, its jumps occur with intensity $1$, and the size of each jump is distributed as $\Exp(1)$. Then conditions of Theorem~\ref{thm:CLP}, as well as Assumption 4, are fulfilled. Knowing the moment generating function of the exponential distribution, we can calculate
$$
K(\la) = -3\la + \la^2 + \frac{2\la}{1 - \la}. 
$$
This function obtains minimal value $-0.0748337$ at $\la_* = 0.141906$. Therefore, the gap process is $V_{\la_*}$-uniformly ergodic with exponential rate of convergence $\vk = 0.0748337$. 
\end{exm}

The next example is when the gap process is not stochastically ordered, but is dominated by a stochastically ordered uniformly ergodic reflected jump-diffusion; we use Corollary~\ref{cor:non-stoch-ordered-RJD}.

\begin{exm} Take a system of two competing Levy particles with independent jumps, governed by measures $\nu_+ = 0$ and $\nu_- = \de_{1}$, with drift and diffusion coefficients 
$$
g_+ = 0,\ g_- = 2,\ a_{++} = a_{--} = 1,\ a_{+-} = 0.
$$
As follows from the results of Section 6, the gap process is a reflected jump-diffusion with $g = -2$, $\si^2 = 2$, and $\nu_x = \de_{|x-1|}$ for $x \in \BR_+$. But $\nu_x \preceq \ol{\nu}_x := \de_{x+1}$ for $x \in \BR_+$, and so 
$$
\ol{K}(x, \la) \equiv \ol{K}(\la) = -2\la + \la^2 + e^{\la} - 1.
$$
This function assumes its minimal value $-0.160516$ at $\la_* = 0.314923$. Therefore, the gap process is $V_{\la_*}$-uniformly ergodic with exponential rate of convergence $\vk = 0.160516$. 
\end{exm}

\section*{Acknowledgements}

This research was partially supported by NSF grants DMS 1007563, DMS 1308340, DMS 1409434, and DMS 1405210. The author thanks Amarjit Budhiraja, Tomoyuki Ichiba, and Robert Lund for useful suggestions.

%\begin{exm} Consider a reflected jump-diffusion with $g(x) = x$, $\si

%\section{Appendix}
%
%\begin{lemma} Under Assumption 2, for every $f \in C_b(\BR_+)$ the function $F : x \mapsto \int_{\BR_+}f(y)\nu_x(dy)$ is also in $C_b(\BR_+)$. 
%\label{mb}
%\end{lemma}
%
%\begin{proof} Continuity follows from the definition of weak convergence of measures. Boundedness follows from $\left|\int_{\mathfrak X}f(y)\nu_x(dy)\right| \le \rho\cdot\norm{f}$. 
%\end{proof}

\end{document}